\documentclass[final,leqno]{siamltex704}
\usepackage{amsmath}
\usepackage{amssymb}
\usepackage{graphicx}
\usepackage[notcite,notref]{showkeys}
\newtheorem{defi}{Definition}[section]
\newtheorem{example}{\bf WG Example}[section]
\newtheorem{algorithm}{Weak Galerkin Algorithm}
\newcommand{\bq}{\begin{equation}}
\newcommand{\eq}{\end{equation}}

\newcommand{\bx}{{\bf x}}

\newcommand{\bv}{{\bf v}}

\def\bn{{\bf n}}
\def\bq{{\bf q}}

\def\3bar{{|\hspace{-.02in}|\hspace{-.02in}|}}

\title{A Weak Galerkin finite element method for second-order elliptic problems}
\author{Junping Wang\thanks{Division of Mathematical Sciences, National
Science Foundation, Arlington, VA 22230 (jwang@\break nsf.gov). The
research of Wang was supported by the NSF IR/D program, while
working at the Foundation. However, any opinion, finding, and
conclusions or recommendations expressed in this material are those
of the author and do not necessarily reflect the views of the
National Science Foundation.} \and Xiu Ye\thanks{Department of
Mathematics, University of Arkansas at Little Rock, Little Rock, AR
72204 (xxye@ualr.edu). This research was supported in part by
National Science Foundation Grant DMS-0813571}}
\begin{document}
\maketitle

\begin{abstract}
In this paper, authors shall introduce a finite element method by
using a weakly defined gradient operator over discontinuous
functions with heterogeneous properties. The use of weak gradients
and their approximations results in a new concept called {\em
discrete weak gradients} which is expected to play important roles
in numerical methods for partial differential equations. This
article intends to provide a general framework for operating
differential operators on functions with heterogeneous properties.
As a demonstrative example, the discrete weak gradient operator is
employed as a building block to approximate the solution of a model
second order elliptic problem, in which the classical gradient
operator is replaced by the discrete weak gradient. The resulting
numerical approximation is called a weak Galerkin (WG) finite
element solution. It can be seen that the weak Galerkin method
allows the use of totally discontinuous functions in the finite
element procedure. For the second order elliptic problem, an optimal
order error estimate in both a discrete $H^1$ and $L^2$ norms are
established for the corresponding weak Galerkin finite element
solutions. A superconvergence is also observed for the weak Galerkin
approximation.
\end{abstract}


\pagestyle{myheadings}

\section{Introduction}

The goal of this paper is to introduce a numerical approximation
technique for partial differential equations based on a new
interpretation of differential operators and their approximations.
To illustrate the main idea, we consider the Dirichlet problem for
second-order elliptic equations which seeks an unknown functions
$u=u(x)$ satisfying
\begin{eqnarray}
-\nabla\cdot (a \nabla u)+\nabla\cdot (b u)+cu &=& f\quad
\mbox{in}\;\Omega,\label{pde}\\
u&=&g\quad \mbox{on}\; \partial\Omega,\label{bc}
\end{eqnarray}
where $\Omega$ is a polygonal or polyhedral domain in
$\mathbb{R}^d\; (d=2,3)$, $a=(a_{ij}(x))_{d\times d}\in
[L^{\infty}(\Omega)]^{d^2}$ is a symmetric matrix-valued function,
$b=(b_i(x))_{d\times 1}$ is a vector-valued function, and $c=c(x)$
is a scalar function on $\Omega$. Assume that the matrix $a$
satisfies the following property: there exists a constant $\alpha
>0$ such that
\begin{equation}\label{matrix}
\alpha\xi^T\xi\leq \xi^T a\xi,\quad\forall \xi\in \mathbb{R}^d.
\end{equation}
For simplicity, we shall concentrate on two-dimensional problems
only (i.e., $d=2$). An extension to higher-dimensional problems is
straightforward.

The standard weak form for (\ref{pde}) and (\ref{bc}) seeks $u\in
H^1(\Omega)$ such that $u=g$ on $\partial\Omega$ and
\begin{eqnarray}\label{weakform}
(a\nabla u, \nabla v)-(b u,\nabla v)+(cu,v)=(f,v)\quad \forall v\in
H_0^1(\Omega),
\end{eqnarray}
where $(\phi,\psi)$ represents the $L^2$-inner product of
$\phi=\phi(x)$ and $\psi=\psi(x)$ -- either vector-valued or
scalar-valued functions. Here $\nabla u$ denotes the gradient of the
function $u=u(x)$, and $\nabla$ is known as the gradient operator.
In the standard Galerkin method (e.g., see \cite{ci, sue}), the
trial space $H^1(\Omega)$ and the test space $H_0^1(\Omega)$ in
(\ref{weakform}) are each replaced by properly defined subspaces of
finite dimensions. The resulting solution in the subspace/subset is
called a Galerkin approximation. A key feature in the Galerkin
method is that the approximating functions are chosen in a way that
the gradient operator $\nabla$ can be successfully applied to them
in the classical sense. A typical implication of this property in
Galerkin finite element methods is that the approximating functions
(both trial and test) are continuous piecewise polynomials over a
prescribed finite element partition for the domain, often denoted by
${\cal T}_h$. Therefore, a great attention has been paid to a
satisfaction of the embedded ``continuity" requirement in the
research of Galerkin finite element methods in existing literature
till recent advances in the development of discontinuous Galerkin
methods. But the interpretation of the gradient operator still lies
in the classical sense for both ``continuous" and ``discontinuous"
Galerkin finite element methods in current existing literature.

In this paper, we will introduce a weak gradient operator defined on
a space of functions with heterogeneous properties. The weak
gradient operator will then be employed to discretize the problem
(\ref{weakform}) through the use of a discrete weak gradient
operator as building bricks. The corresponding finite element method
is called {\em weak Galerkin} method. Details can be found in
Section \ref{section4}.

To explain weak gradients, let $K$ be any polygonal domain with
interior $K^0$ and boundary $\partial K$. A {\em weak function} on
the region $K$ refers to a vector-valued function $v=\{v_0, v_b\}$
such that $v_0\in L^2(K)$ and $v_b\in H^{\frac12}(\partial K)$. The
first component $v_0$ can be understood as the value of $v$ in the
interior of $K$, and the second component $v_b$ is the value of $v$
on the boundary of $K$. Note that $v_b$ may not be necessarily
related to the trace of $v_0$ on $\partial K$ should a trace be
defined. Denote by $W(K)$ the space of weak functions associated
with $K$; i.e.,
\begin{equation}\label{hi.888.new}
W(K) = \{v=\{v_0, v_b \}:\ v_0\in L^2(K),\; v_b\in
H^{\frac12}(\partial K)\}.
\end{equation}

Recall that the dual of $L^2(K)$ can be identified with itself by
using the standard $L^2$ inner product as the action of linear
functionals. With a similar interpretation, for any $v\in W(K)$, the
{\bf weak gradient} of $v$ can be defined as a linear functional
$\nabla_d v$ in the dual space of $H({\rm div},K)$ whose action on
each $q\in H({\rm div},K)$ is given by
\begin{equation}\label{weak-gradient-new}
(\nabla_d v, q) := -\int_K v_0 \nabla\cdot q dK+ \int_{\partial K}
v_b q\cdot\bn ds,
\end{equation}
where $\bn$ is the outward normal direction to $\partial K$. Observe
that for any $v\in W(K)$, the right-hand side of
(\ref{weak-gradient-new}) defines a bounded linear functional on the
normed linear space $H({\rm div}, K)$. Thus, the weak gradient
$\nabla_d v$ is well defined. With the weak gradient operator
$\nabla_d$ being employed in (\ref{weakform}), the trial and test
functions can be allowed to take separate values/definitions on the
interior of each element $T$ and its boundary. Consequently, we are
left with a greater option in applying the Galerkin to partial
differential equations.

Many numerical methods have been developed for the model problem
(\ref{pde})-(\ref{bc}). The existing methods can be classified into
two categories: (1) methods based on the primary variable $u$, and
(2) methods based on the variable $u$ and a flux variable (mixed
formulation). The standard Galerkin finite element methods
(\cite{ci, sue, baker}) and various interior penalty type
discontinuous Galerkin methods (\cite{arnold, abcm, bo, rwg, rwg-1})
are typical examples of the first category. The standard mixed
finite elements (\cite{rt, arnold-brezzi, babuska, brezzi, bf, bdm,
bddf, wang}) and various discontinuous Galerkin methods based on
both variables (\cite{ccps, cs, cgl, jp}) are representatives of the
second category. Due to the enormous amount of publications
available in general finite element methods, it is unrealistic to
list all the key contributions from the computational mathematics
research community in this article. The main intention of the above
citation is to draw a connection between existing numerical methods
with the one that is to be presented in the rest of the Sections.

The weak Galerkin finite element method, as detailed in Section
\ref{section4}, is closely related to the mixed finite element
method (see \cite{rt, arnold-brezzi, babuska, brezzi, bdm, wang})
with a hybridized interpretation of Fraeijs de Veubeke \cite{fdv1,
fdv2}. The hybridized formulation introduces a new term, known as
the Lagrange multiplier, on the boundary of each element. The
Lagrange multiplier is known to approximate the original function
$u=u(x)$ on the boundary of each element. The concept of {\em weak
gradients} shall provide a systematic framework for dealing with
discontinuous functions defined on elements and their boundaries in
a near classical sense. As far as we know, the resulting weak
Galerkin methods and their error estimates are new in many
applications.

\section{Preliminaries and Notations}\label{section2}

We use standard definitions for the Sobolev spaces $H^s(D)$ and
their associated inner products $(\cdot,\cdot)_{s,D}$, norms
$\|\cdot\|_{s,D}$, and seminorms $|\cdot|_{s,D}$ for $s\ge 0$. For
example, for any integer $s\ge 0$, the seminorm $|\cdot|_{s, D}$
is given by
$$
|v|_{s, D} = \left( \sum_{|\alpha|=s} \int_D |\partial^\alpha v|^2 dD
\right)^{\frac12},
$$
with the usual notation
$$
\alpha=(\alpha_1, \alpha_2), \quad |\alpha| = \alpha_1+\alpha_2,\quad
\partial^\alpha =\partial_{x_1}^{\alpha_1} \partial_{x_2}^{\alpha_2}.
$$
The Sobolev norm $\|\cdot\|_{m,D}$ is given by
$$
\|v\|_{m, D} = \left(\sum_{j=0}^m |v|^2_{j,D} \right)^{\frac12}.
$$

The space $H^0(D)$ coincides with $L^2(D)$, for which the norm and
the inner product are denoted by $\|\cdot \|_{D}$ and
$(\cdot,\cdot)_{D}$, respectively. When $D=\Omega$, we shall drop
the subscript $D$ in the norm and inner product notation. The space
$H({\rm div};\Omega)$ is defined as the set of vector-valued
functions on $\Omega$ which, together with their divergence, are
square integrable; i.e.,
\[
H({\rm div}; \Omega)=\left\{ \bv: \ \bv\in [L^2(\Omega)]^2,
\nabla\cdot\bv \in L^2(\Omega)\right\}.
\]
The norm in $H({\rm div}; \Omega)$ is defined by
$$
\|\bv\|_{H({\rm div}; \Omega)} = \left( \|\bv\|^2 + \|\nabla
\cdot\bv\|^2\right)^{\frac12}.
$$

\section{A Weak Gradient Operator and Its Approximation}\label{section3}

The goal of this section is to introduce a weak gradient operator
defined on a space of functions with heterogeneous properties. The
weak gradient operator will then be employed to discretize partial
differential equations. To this end, let $K$ be any polygonal domain
with interior $K^0$ and boundary $\partial K$. A {\em weak function}
on the region $K$ refers to a vector-valued function $v=\{v_0,
v_b\}$ such that $v_0\in L^2(K)$ and $v_b\in H^{\frac12}(\partial
K)$. The first component $v_0$ can be understood as the value of $v$
in the interior of $K$, and the second component $v_b$ is the value
of $v$ on the boundary of $K$. Note that $v_b$ may not be
necessarily related to the trace of $v_0$ on $\partial K$ should a
trace be well defined. Denote by $W(K)$ the space of weak functions
associated with $K$; i.e.,
\begin{equation}\label{hi.888}
W(K) = \{v=\{v_0, v_b \}:\ v_0\in L^2(K),\; v_b\in
H^{\frac12}(\partial K)\}.
\end{equation}

\medskip

\begin{defi}
The dual of $L^2(K)$ can be identified with itself by using the
standard $L^2$ inner product as the action of linear functionals.
With a similar interpretation, for any $v\in W(K)$, the {\bf weak
gradient} of $v$ is defined as a linear functional $\nabla_d v$ in
the dual space of $H({\rm div},K)$ whose action on each $q\in H({\rm
div},K)$ is given by
\begin{equation}\label{weak-gradient}
(\nabla_d v, q) := -\int_K v_0 \nabla\cdot q dK+ \int_{\partial K}
v_b q\cdot\bn ds,
\end{equation}
where $\bn$ is the outward normal direction to $\partial K$.
\end{defi}

\medskip

Note that for any $v\in W(K)$, the right-hand side of
(\ref{weak-gradient}) defines a bounded linear functional on the
normed linear space $H({\rm div}, K)$. Thus, the weak gradient
$\nabla_d v$ is well defined. Moreover, if the components of $v$ are
restrictions of a function $u\in H^1(K)$ on $K^0$ and $\partial K$,
respectively, then we would have
$$
-\int_K v_0 \nabla\cdot q dK+ \int_{\partial K} v_b q\cdot\bn ds =
-\int_K u \nabla\cdot q dK+ \int_{\partial K} u q\cdot\bn ds =
\int_K \nabla u \cdot q dK.
$$
It follows that $\nabla_d v= \nabla u$ is the classical gradient of
$u$.

\bigskip

Next, we introduce a discrete weak gradient operator by defining
$\nabla_d$ in a polynomial subspace of $H({\rm div}, K)$. To this
end, for any non-negative integer $r\ge 0$, denote by $P_{r}(K)$ the
set of polynomials on $K$ with degree no more than $r$. Let
$V(K,r)\subset [P_{r}(K)]^2$ be a subspace of the space of
vector-valued polynomials of degree $r$. A discrete weak gradient
operator, denoted by $\nabla_{d,r}$, is defined so that
$\nabla_{d,r} v \in V(K,r)$ is the unique solution of the following
equation
\begin{equation}\label{discrete-weak-gradient}
\int_K \nabla_{d,r} v\cdot q dK = -\int_K v_0 \nabla\cdot q dK+
\int_{\partial K} v_b q\cdot\bn ds,\qquad \forall q\in V(K,r).
\end{equation}

It is not hard to see that the discrete weak gradient operator
$\nabla_{d,r}$ is a Galerkin-type approximation of the weak gradient
operator $\nabla_d$ by using the polynomial space $V(K,r)$.

The classical gradient operator $\nabla=(\partial_{x_1},
\partial_{x_2})$ should be applied to functions with certain smoothness in
the design of numerical methods for partial differential equations.
For example, in the standard Galerkin finite element method, such a
``smoothness" often refers to continuous piecewise polynomials over
a prescribed finite element partition. With the weak gradient
operator as introduced in this section, derivatives can be taken for
functions without any continuity across the boundary of each
triangle. Thus, the concept of weak gradient allows the use of
functions with heterogeneous properties in approximation.

Analogies of weak gradient can be established for other differential
operators such as divergence and curl operators. Details for weak
divergence and weak curl operators and their applications in
numerical methods will be given in forthcoming papers.

\section{A Weak Galerkin Finite Element Method}\label{section4}
The goal of this section is to demonstrate how discrete weak
gradients be used in the design of numerical schemes that
approximate the solution of partial differential equations. For
simplicity, we take the second order elliptic equation (\ref{pde})
as a model for discussion. With the Dirichlet boundary condition
(\ref{bc}), the standard weak form seeks $u\in H^1(\Omega)$ such
that $u=g$ on $\partial\Omega$ and
\begin{eqnarray}\label{weakform-new}
(a\nabla u, \nabla v)-(b u,\nabla v)+(cu,v)=(f,v)\quad \forall v\in
H_0^1(\Omega).
\end{eqnarray}

Let ${\cal T}_h$ be a triangular partition of the domain $\Omega$
with mesh size $h$. Assume that the partition ${\cal T}_h$ is shape
regular so that the routine inverse inequality in the finite element
analysis holds true (see \cite{ci}). In the general spirit of
Galerkin procedure, we shall design a weak Galerkin method for
(\ref{weakform-new}) by following two basic principles: {\em (1)
replace $H^1(\Omega)$ by a space of discrete weak functions defined
on the finite element partition ${\cal T}_h$ and the boundary of
triangular elements; (2) replace the classical gradient operator by
a discrete weak gradient operator $\nabla_{d,r}$ for weak functions
on each triangle $T$.} Details are to be presented in the rest of
this section.

\bigskip

For each $T\in {\cal T}_h$, Denote by $P_j(T^0)$ the set of
polynomials on $T^0$ with degree no more than $j$, and
$P_\ell(\partial T)$ the set of polynomials on $\partial T$ with
degree no more than $\ell$ (i.e., polynomials of degree $\ell$ on
each line segment of $\partial T$). A {\em discrete weak function}
$v=\{v_0, v_b\}$ on $T$ refers to a weak function $v=\{v_0, v_b\}$
such that $v_0\in P_j(T^0)$ and $v_b\in P_\ell(\partial T)$ with
$j\ge 0$ and $\ell \ge 0$. Denote this space by $W(T, j, \ell)$,
i.e.,
$$
W(T,j,\ell) := \left\{v=\{v_0, v_b\}:\  v_0\in P_j(T^0), v_b\in
P_\ell(\partial T)\right\}.
$$
The corresponding finite element space would be defined by patching
$W(T,j,\ell)$ over all the triangles $T\in {\cal T}_h$. In other
words, the weak finite element space is given by
\begin{equation}\label{weak-fes}
S_h(j,\ell) :=\left\{ v=\{v_0, v_b\}:\ \{v_0, v_b\}|_{T}\in
W(T,j,\ell), \forall T\in {\cal T}_h \right\}.
\end{equation}
Denote by $S_h^0(j,\ell)$ the subspace of $S_h(j,\ell)$ with
vanishing boundary values on $\partial\Omega$; i.e.,
\begin{equation}\label{weak-fes-homo}
S_h^0(j,\ell) :=\left\{ v=\{v_0, v_b\}\in S_h(j,\ell),
{v_b}|_{\partial T\cap \partial\Omega}=0, \ \forall T\in {\cal T}_h
\right\}.
\end{equation}

According to (\ref{discrete-weak-gradient}), for each $v=\{v_0,
v_b\} \in S_h(j,\ell)$, the discrete weak gradient of $v$ on each
element $T$ is given by the following equation:
\begin{equation}\label{discrete-weak-gradient-new}
\int_T \nabla_{d,r} v\cdot q dT = -\int_T v_0 \nabla\cdot q dT+
\int_{\partial T} v_b q\cdot\bn ds,\qquad \forall q\in V(T,r).
\end{equation}
Note that no specific examples of the approximating space $V(T,r)$
have been mentioned, except that $V(T,r)$ is a subspace of the set
of vector-valued polynomials of degree no more than $r$ on $T$.

For any $w,v \in S_h(j,\ell)$, we introduce the following bilinear
form
\begin{equation}\label{linearform-a}
a(w,v)=(a\nabla_{d,r} w,\;\nabla_{d,r} v)-(b u_0,\nabla_{d,r}
v)+(cu_0,v_0),
\end{equation}
where
\begin{eqnarray*}
(a\nabla_{d,r} w,\;\nabla_{d,r} v)&=&\int_\Omega a\nabla_{d,r} w
\cdot\nabla_{d,r}v d\Omega,\\
(bw_0,\;\nabla_{d,r} v)&=&\int_{\Omega} bu_0\cdot \nabla_{d,r}v
d\Omega,\\
(cw_0,v_0)&=&\int_\Omega cw_0 v_0 d\Omega.
\end{eqnarray*}

\begin{algorithm}
A numerical approximation for (\ref{pde}) and (\ref{bc}) can be
obtained by seeking $u_h=\{u_0,u_b\}\in S_h(j,\ell)$ satisfying
$u_b= Q_b g$ on $\partial \Omega$ and the following equation:
\begin{equation}\label{WG-fem}
a(u_h,v)=(f,\;v_0), \quad\forall\ v=\{v_0, v_b\}\in S_h^0(j,\ell),
\end{equation}
where $Q_b g$ is an approximation of the boundary value in the
polynomial space $P_\ell(\partial T\cap \partial\Omega)$. For
simplicity, $Q_b g$ shall be taken as the standard $L^2$ projection
for each boundary segment; other approximations of the boundary
value $u=g$ can also be employed in (\ref{WG-fem}).
\end{algorithm}

\section{Examples of Weak Galerkin Method with Properties}

Although the weak Galerkin scheme (\ref{WG-fem}) is defined for
arbitrary indices $j, \ell$, and $r$, the method can be shown to
produce good numerical approximations for the solution of the
original partial differential equation only with a certain
combination of their values. For one thing, there are at least two
prominent properties that the discrete gradient operator
$\nabla_{d,r}$ should possess in order for the weak Galerkin method
to work well. These two properties are:
\begin{enumerate}
\item[\bf P1:] For any $v\in S_h(j,\ell)$, if $\nabla_{d,r} v=0$ on $T$, then
one must have $v\equiv constant$ on $T$. In other words,
$v_0=v_b=constant$ on $T$;
\item[\bf P2:] Let $u\in H^m(\Omega) (m\ge 1)$ be a smooth function on
$\Omega$, and $Q_h u$ be a certain interpolation/projection of $u$
in the finite element space $S_h(j,\ell)$. Then, the discrete weak
gradient of $Q_h u$ should be a good approximation of $\nabla u$.
\end{enumerate}

\bigskip
The following are two examples of weak finite element spaces that
fit well into the numerical scheme (\ref{WG-fem}).
\medskip
\begin{example}\label{wg-example1}
In this example, we take $\ell=j+1, r=j+1$, and
$V(T,j+1)=\left[P_{j+1}(T)\right]^2$, where $j\ge 0$ is any
non-negative integer. Denote by $S_h(j,j+1)$ the corresponding
finite element space. More precisely, the finite element space
$S_h(j,j+1)$ consists of functions $v=\{v_0, v_b\}$ where $v_0$ is a
polynomial of degree no more than $j$ in $T^0$, and $v_b$ is a
polynomial of degree no more than $j+1$ on $\partial T$. The space
$V(T,r)$ used to define the discrete weak gradient operator
$\nabla_{d,r}$ in (\ref{discrete-weak-gradient-new}) is given as
vector-valued polynomials of degree no more than $j+1$ on $T$.
\end{example}

\medskip
\begin{example}\label{wg-example2}
In the second example, we take $\ell=j, r=j+1$, and
$V(T,r=j+1)=\left[P_{j}(T)\right]^2 + \widehat P_j(T) \bx$, where
$\bx=(x_1,x_2)^T$ is a column vector and $\widehat P_j(T)$ is the
set of homogeneous polynomials of order $j$ in the variable $\bx$.
Denote by $S_h(j,j)$ the corresponding finite element space. Note
that the space $V(T,r)$ that was used to define a discrete weak
gradient is in fact the usual Raviart-Thomas element \cite{rt} of
order $j$ for the vector component.
\end{example}

\bigskip
Let us demonstrate how the two properties {\bf P1} and {\bf P2} are
satisfied with the two examples given as above. For simplicity, we
shall present results only for {\bf WG Example \ref{wg-example1}}.
The following result addresses a satisfaction of the property {\bf
P1}.
\medskip

\begin{lemma}\label{lemma-zero}
For any $v=\{v_0, v_b\}\in W(T, j, j+1)$, let $\nabla_{d,j+1} v$ be
the discrete weak gradient of $v$ on $T$ as defined in
(\ref{discrete-weak-gradient-new}) with
$V(T,r)=\left[P_{j+1}(T)\right]^2$. Then, $\nabla_{d,j+1} v =0$
holds true on $T$ if and only if $v=constant$ (i.e.,
$v_0=v_b=constant$).
\end{lemma}

\begin{proof}
It is trivial to see from (\ref{discrete-weak-gradient-new}) that if
$v=constant$ on $T$, then the right-hand side of
(\ref{discrete-weak-gradient-new}) would be zero for any $q\in
V(T,j+1)$. Thus, we must have $\nabla_{d,j+1} v =0$.

Now assume that $\nabla_{d,j+1} v =0$. It follows from
(\ref{discrete-weak-gradient-new}) that
\begin{equation}\label{discrete-weak-gradient-newv}
-\int_T v_0 \nabla\cdot q dT+ \int_{\partial T} v_b q\cdot\bn
ds=0,\qquad \forall q\in V(T,j+1).
\end{equation}
Let $\bar{v}_0$ be the average of $v_0$ over $T$. Using the results
of \cite{bdm}, there exists a vector-valued polynomial $q_1\in
V(T,j+1)=[P_{j+1}(T)]^2$ such that $q_1\cdot\bn=0$ on $\partial T$
and $\nabla\cdot q_1 = v_0 - \bar{v}_0$. With $q=q_1$ in
(\ref{discrete-weak-gradient-newv}), we arrive at $\int_T
(v_0-\bar{v}_0)^2 dT=0$. It follows that $v_0=\bar{v}_0$, and
(\ref{discrete-weak-gradient-newv}) can be rewritten as
\begin{equation}\label{discrete-weak-gradient-newvw}
\int_{\partial T} (v_b-v_0) q\cdot\bn ds=0,\qquad \forall q\in
V(T,j+1).
\end{equation}
Now since $v_b-v_0\in P_{j+1}(\partial T)$, then one may select a
$q\in V(T,j+1)=[P_{j+1}(T)]^2$ such that
$$
\int_{\partial T} \phi q\cdot\bn ds = \int_{\partial T} \phi
(v_b-v_0)ds,\qquad \forall \phi\in P_{j+1}(\partial T),
$$
which, together with (\ref{discrete-weak-gradient-newvw}) and
$\phi=v_b-v_0$ yields
$$
\int_{\partial T} (v_b-v_0)^2ds = 0.
$$
The last equality implies $v_b=v_0=constant$, which completes a
proof of the lemma.
\end{proof}

\medskip

To verify property {\bf P2}, let $u\in H^1(T)$ be a smooth function
on $T$. Denote by  $Q_h u=\{Q_0 u,\;Q_bu\}$ the $L^2$ projection
onto $P_j(T^0)\times P_{j+1}(\partial T)$. In other words, on each
element $T$, the function $Q_0 u$ is defined as the $L^2$ projection
of $u$ in $P_j(T)$ and on $\partial T$, $Q_b u$ is the $L^2$
projection in $P_{j+1}(\partial T)$. Furthermore, let $R_h$ be the
local $L^2$ projection onto $V(T,j+1)$. According to the definition
of $\nabla_{d,j+1}$, the discrete weak gradient function
$\nabla_{d,j+1}(Q_hu)$ is given by the following equation:
\begin{equation}\label{discrete-weak-gradient-hi}
\int_T \nabla_{d,j+1}(Q_h u) \cdot q dT = -\int_T (Q_0 u)
\nabla\cdot q dT+ \int_{\partial T} (Q_b u) q\cdot\bn ds,\quad
\forall q\in V(K,j+1).
\end{equation}
Since $Q_0$ and $Q_b$ are $L^2$-projection operators, then the
right-hand side of (\ref{discrete-weak-gradient-hi}) is given by
\begin{eqnarray*}
-\int_T (Q_0 u) \nabla\cdot q dT+ \int_{\partial T} (Q_b u)
q\cdot\bn ds &=& -\int_T u \nabla\cdot q dT+ \int_{\partial T} u
q\cdot\bn ds \\
&=& \int_T (\nabla u)\cdot q dT = \int_T (R_h \nabla u)\cdot q dT.
\end{eqnarray*}
Thus, we have derived the following useful identity:
\begin{equation}\label{4.88}
\nabla_{d,j+1}(Q_h u) =R_h (\nabla u),\qquad \forall u\in H^1(T).
\end{equation}
The above identity clearly indicates that $\nabla_{d,j+1}(Q_h u)$ is
an excellent approximation of the classical gradient of $u$ for any
$u\in H^1(T)$. Thus, it is reasonable to believe that the weak
Galerkin finite element method shall provide a good numerical scheme
for the underlying partial differential equations.

\section{Mass Conservation of Weak Galerkin}

The second order elliptic equation (\ref{pde}) can be rewritten in a
conservative form as follows:
$$
\nabla \cdot q + cu = f, \quad q=-a\nabla u + bu.
$$
Let $T$ be any control volume. Integrating the first equation over
$T$ yields the following integral form of mass conservation:
\begin{equation}\label{conservation.01}
\int_{\partial T} q\cdot \bn ds + \int_T cu dT = \int_T f dT.
\end{equation}
We claim that the numerical approximation from the weak Galerkin
finite element method for (\ref{pde}) retains the mass conservation
property (\ref{conservation.01}) with a numerical flux $q_h$. To
this end, for any given $T\in {\cal T}_h$, we chose in
(\ref{WG-fem}) a test function $v=\{v_0, v_b=0\}$ so that $v_0=1$ on
$T$ and $v_0=0$ elsewhere. Using the relation (\ref{linearform-a}),
we arrive at
\begin{equation}\label{mass-conserve.08}
\int_T a\nabla_{d,r} u_h\cdot \nabla_{d,r}v dT - \int_T b u_{0}
\cdot\nabla_{d,r}v dT + \int_T c u_{0} dT = \int_T f dT.
\end{equation}
Using the definition (\ref{discrete-weak-gradient-new}) for
$\nabla_{d,r}$, one has
\begin{eqnarray}
\int_T a\nabla_{d,r} u_h\cdot \nabla_{d,r}v dT &=& \int_T
R_h(a\nabla_{d,r} u_h)\cdot \nabla_{d,r}v dT \nonumber\\
&=& - \int_T \nabla\cdot R_h(a\nabla_{d,r} u_h) dT \nonumber\\
&=& - \int_{\partial T} R_h(a\nabla_{d,r}u_h)\cdot\bn ds
\label{conserv.88}
\end{eqnarray}
and
\begin{eqnarray}
\int_T b u_{0} \cdot\nabla_{d,r}v dT &=& \int_T R_h(b u_{0})
\cdot\nabla_{d,r}v dT\nonumber\\
&=& -\int_T \nabla\cdot R_h(b u_{0})dT\nonumber\\
&=& -\int_{\partial T} R_h(b u_{0})\cdot\bn ds\label{conserve.89}
\end{eqnarray}
Now substituting (\ref{conserve.89}) and (\ref{conserv.88}) into
(\ref{mass-conserve.08}) yields
\begin{equation}\label{mass-conserve.09}
\int_{\partial T} R_h\left(-a\nabla_{d,r}u_h + b u_{0}
\right)\cdot\bn ds +\int_T c u_{0} dT = \int_T f dT,
\end{equation}
which indicates that the weak Galerkin method conserves mass with a
numerical flux given by
$$
q_h\cdot\bn =R_h\left(-a\nabla_{d,r}u_h + b u_{0} \right)\cdot\bn.
$$
The numerical flux $q_h\cdot\bn$ can be verified to be continuous
across the edge of each element $T$ through a selection of the test
function $v=\{v_0,v_b\}$ so that $v_0\equiv 0$ and $v_b$ arbitrary.

\section{Existence and Uniqueness for Weak Galerkin Approximations}

Assume that $u_h$ is a weak Galerkin approximation for the problem
(\ref{pde}) and (\ref{bc}) arising from (\ref{WG-fem}) by using the
finite element space $S_h(j,j+1)$ or $S_h(j,j)$. The goal of this
section is to derive a uniqueness and existence result for $u_h$.
For simplicity, details are only presented for the finite element
space $S_h(j,j+1)$; the result can be extended to $S_h(j,j)$ without
any difficulty.

\medskip
First of all, let us derive the following analogy of G\aa{rding's}
inequality.
\medskip
\begin{lemma}
Let $S_h(j,\ell)$ be the weak finite element space defined in
(\ref{weak-fes}) and $a(\cdot,\cdot)$ be the bilinear form given in
(\ref{linearform-a}). There exists a constant $K$ and $\alpha_1$
satisfying
\begin{equation}\label{garding}
a(v,v)+K(v_0, v_0)\ge \alpha_1(\|\nabla_{d,r}v\|^2+\|v_0\|^2),
\end{equation}
for all $v\in S_{h}(j,\ell)$.
\end{lemma}
\begin{proof}
Let $B_1=\|b\|_{L^\infty(\Omega)}$ and
$B_2=\|c\|_{L^\infty(\Omega)}$ be the $L^\infty$ norm of the
coefficients $b$ and $c$, respectively. Since
\begin{eqnarray*}
|(bv_0, \nabla_{d,r} v)|&\le& B_1\|\nabla_{d,r} v\|\ \|v_0\|,\\
|(cv_0, v_0)|&\leq & B_2 \|v_0\|^2,
\end{eqnarray*}
then it follows from (\ref{linearform-a}) that there exists a
constant $K$ and $\alpha_1$ such that
\begin{eqnarray*}
a(v,v)+K(v_0,v_0)&\ge& \alpha\|\nabla_{d,r}v\|^2-B_1\|\nabla_{d,r}v\|\|v_0\|+(K-B_2)\|v_0\|^2\\
&\ge& \alpha_1(\|\nabla_{d,r}v\|^2+\|v_0\|^2),
\end{eqnarray*}
which completes the proof.
\end{proof}

\medskip

For simplicity of notation, we shall drop the subscript $r$ in the
discrete weak gradient operator $\nabla_{d,r}$ from now on. Readers
should bear in mind that $\nabla_d$ refers to a discrete weak
gradient operator defined by using the setups of either Example
\ref{wg-example1} or Example \ref{wg-example2}. In fact, for these
two examples, one may also define a projection $\Pi_h$ such that
$\Pi_h\bq\in H({\rm div},\Omega)$, and on each $T\in {\cal T}_h$,
one has $\Pi_h\bq \in V(T, r=j+1)$ and the following identity
$$
(\nabla\cdot\bq,\;v_0)_T=(\nabla\cdot\Pi_h\bq,\;v_0)_T, \qquad
\forall v_0\in P_j(T^0).
$$
The following result is based on the above property of $\Pi_h$.

\medskip
\begin{lemma}
For any $\bq\in  H({\rm div},\Omega)$,  we have
\begin{equation}\label{4.200}
\sum_{T\in {\cal T}_h}(-\nabla\cdot\bq, \;v_0)_T=\sum_{T\in {\cal
T}_h}(\Pi_h\bq, \;\nabla_dv)_T,
\end{equation}
for all $v=\{v_0,v_b\}\in S^0_h(j,j+1)$.
\end{lemma}
\begin{proof}
The definition of $\Pi_h$ and the definition of $\nabla_d v$ imply ,
\begin{eqnarray*}
\sum_{T\in {\cal T}_h}(-\nabla\cdot\bq, \;v_0)_T&=&\sum_{T\in {\cal T}_h}(-\nabla\cdot \Pi_h\bq, \;v_0)_T\\
&=&\sum_{T\in {\cal T}_h}(\Pi_h\bq, \nabla_d v)_T - \sum_{T\in {\cal
T}_h} \langle v_b,
\Pi_h \bq\cdot \bn\rangle_{\partial T} \\
&=&\sum_{T\in {\cal T}_h}(\Pi_h \bq, \nabla_d v)_T.
\end{eqnarray*}
Here we have used the fact that $\Pi_h \bq\cdot \bn$ is continuous
across each interior edge and $v_b=0$ on $\partial \Omega$. This
completes the proof.
\end{proof}

\medskip

\begin{lemma}\label{approx}
For $u\in H^{1+s}(\Omega)$ with $s>0$, we have
\begin{eqnarray}
\|\Pi_h(a\nabla u)-a\nabla_d(Q_h u)\|&\le& Ch^s\|u\|_{1+s},\label{a2}\\
\|\nabla u-\nabla_d(Q_hu)\|&\le&Ch^s\|u\|_{1+s}.\label{a3}
%\|\Pi_h(bu)-b(Q_0u)\|&\le&Ch^{1+s}\|u\|_{1+s}.\label{a4}
\end{eqnarray}

\end{lemma}
\begin{proof} Since from (\ref{4.88}) we have $\nabla_d(Q_h u) =
R_h(\nabla u)$, then
$$
\|\Pi_h(a\nabla u)-a\nabla_d(Q_h u)\| = \|\Pi_h(a\nabla
u)-aR_h(\nabla u)\|.
$$
Using the triangle inequality and the definition of $\Pi_h$ and
$R_h$, we have
\begin{eqnarray*}
\|\Pi_h(a\nabla u)-aR_h(\nabla u)\|&\le&\|\Pi_h(a\nabla u)-a\nabla u\|+\|a\nabla u-aR_h(\nabla u)\|\\
&\le& Ch^s\|u\|_{1+s}.
\end{eqnarray*}
The estimate (\ref{a3}) can be derived in a similar way. This
completes a proof of the lemma.
\end{proof}

\bigskip

We are now in a position to establish a solution uniqueness and
existence for the weak Galerkin method (\ref{WG-fem}). It suffices
to prove that the solution is unique. To this end, let $e\in
S_h^0(j,j+1)$ be a discrete weak function satisfying
\begin{equation}\label{uniq}
a(e,v)=0,\qquad\forall v=\{v_0, v_b\}\in S_h^0(j,j+1).
\end{equation}
The goal is to show that $e\equiv 0$ by using a duality approach
similar to what Schatz \cite{schatz} did for the standard Galerkin
finite element methods.
\medskip

\begin{lemma}\label{L2byH1}
Let $e=\{e_0, e_b\}\in S_h^0(j,j+1)$ be a discrete weak function
satisfying (\ref{uniq}). Assume that the dual of (\ref{pde}) with
homogeneous Dirichlet boundary condition has the $H^{1+s}$
regularity ($s\in (0,1]$). Then, there exists a constant $C$ such
that
\begin{equation}\label{dual-1}
\|e_0\|\le Ch^s\|\nabla_d e\|,
\end{equation}
provided that the mesh size $h$ is sufficient small, but a fixed
constant.
\end{lemma}

\begin{proof}
Consider the following dual problem: Find $w\in H^1(\Omega)$ such
that
\begin{eqnarray}
-\nabla\cdot (a \nabla w)-b\cdot\nabla w+ cw &=&e_0 \quad
\mbox{in}\;\Omega\label{dual1}\\
w&=&0\quad \mbox{on}\; \partial\Omega,\label{dual1-BC}
\end{eqnarray}
The assumption of $H^{1+s}$ regularity implies that $w\in
H^{1+s}(\Omega)$ and there is a constant $C$ such that
\begin{equation}\label{reg1}
\|w\|_{1+s}\le C\|e_0\|.
\end{equation}
Testing (\ref{dual1}) against $e_0$ and then using (\ref{4.200})
lead to
\begin{eqnarray*}
\|e_0\|^2&=&(-\nabla\cdot (a \nabla w),\;e_0)-(b\cdot\nabla w,\;e_0)+ (cw,\;e_0)\\
&=&(\Pi_h(a\nabla w),\;\nabla_d e)-(\nabla w,\;be_0)+ (cw,\;e_0)\\
&=&(\Pi_h(a\nabla w)-a\nabla_d(Q_h w),\;\nabla_d e)+(a\nabla_d(Q_h w),\;\nabla_d e)\\
& &-(\nabla w-\nabla_d(Q_hw),\;be_0))-(\nabla_d(Q_hw),\;be_0)\\
& &+(cw-c(Q_0w),\;e_0)+(Q_0w,\;ce_0).
\end{eqnarray*}
The sum of the second, forth and sixth term on the right hand side
of the above equation equals $a(e, Q_hw)=0$ due to (\ref{uniq}).
Therefore, it follows from Lemma \ref{approx} that
\begin{eqnarray*}
\|e_0\|^2&=&(\Pi_h(a\nabla w)-a\nabla_d(Q_h w),\;\nabla_d e)-(\nabla w-\nabla_d(Q_hw),\;be_0)\\
& &+(c(w-Q_0w),\;e_0)\\
&\le&Ch^s\|w\|_{1+s}\left(\|\nabla_de\|+\|e_0\|\right) + C h \|w\|_1
\; \|e_0\|.
\end{eqnarray*}
Using the $H^{1+s}$-regularity assumption (\ref{reg1}), we arrive at
$$
\|e_0\|^2 \leq C h^s\|e_0\|\left(\|\nabla_de\|+\|e_0\|\right),
$$
which leads to
$$
\|e_0\| \leq C h^s\left(\|\nabla_de\|+\|e_0\|\right).
$$
Thus, when $h$ is sufficiently small, one would obtain the desired
estimate (\ref{dual-1}). This completes the proof.
\end{proof}

\medskip

\begin{theorem}
Assume that the dual of (\ref{pde}) with homogeneous Dirichlet
boundary condition has $H^{1+s}$-regularity for some $s\in (0,1]$.
The weak Gakerkin finite element method defined in (\ref{WG-fem})
has a unique solution in the finite element spaces $S_h(j,j+1)$ and
$S_j(j,j)$ if the meshsize $h$ is sufficiently small, but a fixed
constant.
\end{theorem}

\bigskip

\begin{proof} Observe that uniqueness is equivalent to existence for the solution of
(\ref{WG-fem}) since the number of unknowns is the same as the
number of equations. To prove a uniqueness, let $u^{(1)}_h$ and
$u^{(2)}_h$ be two solutions of (\ref{WG-fem}). By letting
$e=u^{(1)}_h- u^{(2)}_h$ we see that (\ref{uniq}) is satisfied. Now
we have from the G{\aa}rding's inequality (\ref{garding}) that
$$
a(e,e) + K \|e_0\| \ge \alpha_1\left( \|\nabla_d e\| +
\|e_0\|\right).
$$
Thus, it follows from the estimate (\ref{dual-1}) of Lemma
\ref{L2byH1} that
$$
\alpha_1\left( \|\nabla_d e\| + \|e_0\|\right)\leq CK h^s \|\nabla_d
e\|
$$
for $h$ being sufficiently small. Now chose $h$ small enough so that
$CKh^s\leq \frac{\alpha_1}{2}$. Thus,
$$
\|\nabla_d e\| + \|e_0\|\ = 0,
$$
which, together with Lemma \ref{lemma-zero}, implies that $e$ is a
constant and $e_0=0$. This shows that $e=0$ and consequently,
$u^{(1)}_h=u^{(2)}_h$.
\end{proof}

\section{Error Analysis}
The goal of this section is to derive some error estimate for the
weak Galerkin finite element method (\ref{WG-fem}). We shall follow
the usual approach in the error analysis: (1) investigating the
difference between the weak finite element approximation $u_h$ with
a certain interpolation/projection of the exact solution through an
error equation, (2) using a duality argument to analyze the error in
the $L^2$ norm.

Let us begin with the derivation of an error equation for the weak
Galerkin approximation $u_h$ and the $L^2$ projection of the exact
solution $u$ in the weak finite element space $S_h(j,j+1)$. Recall
that the $L^2$ projection is denoted by $Q_h u \equiv \{Q_0u, Q_b
u\}$, where $Q_0$ denotes the local $L^2$ projection onto $P_j(T)$
and $Q_b$ is the local $L^2$ projection onto $P_{j+1}(\partial T)$
on each triangular element $T\in {\cal T}_h$. Let $v=\{v_0,v_b\}\in
S_h^0(j,j+1)$ be any test function. By testing (\ref{pde}) against
the first component $v_0$ and using (\ref{4.200}) we arrive at
\begin{eqnarray*}
(f, v_0) &=& \sum_{T\in {\cal T}_h}(-\nabla\cdot (a\nabla u), \;v_0)_T
+(\nabla\cdot (b u),\; v_0)+(cu, \;v_0)\nonumber\\
&=&(\Pi_h(a\nabla u),\; \nabla_dv)-(\Pi_h(bu),\; \nabla_d v)+(cu,\; v_0).
\end{eqnarray*}
Adding and subtracting the term $a(Q_hu, v)\equiv
(a\nabla_d(Q_hu),\;\nabla_d
v)-(b(Q_0u),\;\nabla_dv)+(c(Q_0u),\;v_0)$ on the right hand side of
the above equation and then using (\ref{4.88}) we obtain
\begin{eqnarray}
(f, v_0)&=&(a\nabla_d(Q_hu),\;\nabla_d v)-(bQ_0u,\;\nabla_dv)+(cQ_0u,\;v_0)\label{true}\\
& &+(\Pi_h(a\nabla u)-aR_h(\nabla u),\;
\nabla_dv)\nonumber\\
& & -(\Pi_h(bu)-bQ_0u,\; \nabla_d v)+(c(u-Q_0u),\; v_0),\nonumber
\end{eqnarray}
which can be rewritten as
\begin{eqnarray}
a(u_h, v)&=&a(Q_hu,v)+(\Pi_h(a\nabla u)-aR_h(\nabla u),\;
\nabla_dv)\nonumber\\
& & -(\Pi_h(bu)-bQ_0u,\; \nabla_d v)+(c(u-Q_0u),\; v_0).\nonumber
\end{eqnarray}
It follows that
\begin{eqnarray}
a(u_h-Q_hu,\;v)&=&(\Pi_h(a\nabla u)-aR_h(\nabla u),\; \nabla_dv)\nonumber\\
& &-(\Pi_h(bu)-bQ_0u,\; \nabla_d v)+(c(u-Q_0u),\; v_0).\label{diff}
\end{eqnarray}
The equation (\ref{diff}) shall be called the {\em error equation}
for the weak Galerkin finite element method (\ref{WG-fem}).

\subsection{An estimate in a discrete $H^1$-norm}

We begin with the following lemma which provides an estimate for the
difference between the weak Galerkin approximation $u_h$ and the
$L^2$ projection of the exact solution of the original problem.

\medskip

\begin{lemma}\label{h1-error}
Let $u\in H^{1}(\Omega)$ be the solution of (\ref{pde}) and
(\ref{bc}). Let $u_h\in S_h(j,j+1)$ be the weak Galerkin
approximation of $u$ arising from (\ref{WG-fem}). Denote by
$e_h:=u_h-Q_h u$ the difference between the weak Galerkin
approximation and the $L^2$ projection of the exaction solution
$u=u(x_1,x_2)$. Then there exists a constant $C$ such that
\begin{eqnarray}
\frac{\alpha_1}{2}(\|\nabla_d(e_h)\|^2+\|e_{h,0}\|^2) &\le&
C\left(\|\Pi_h(a\nabla u)-aR_h(\nabla u)\|^2 +\|c(u-Q_0u)\|^2
\right.\nonumber\\
& & +\left.
\|\Pi_h(bu)-bQ_0u\|^2\right)+K\|u_0-Q_0u\|^2.\label{H1errorestimate}
\end{eqnarray}
\end{lemma}
\begin{proof} Substituting $v$ in (\ref{diff}) by $e_h:=u_h-Q_hu$
and using the usual Cauchy-Schwarz inequality we arrive at
\begin{eqnarray*}
a(e_h,\;e_h)&=&(\Pi_h(a\nabla u)-aR_h(\nabla u),\; \nabla_d (u_h-Q_hu))\\
& &-(\Pi_h(bu)-bQ_0u,\; \nabla_d (u_h-Q_hu))+(c(u-Q_0u),\; u_0-Q_0u)\\
&\le&\|\Pi_h(a\nabla u)-aR_h(\nabla u)\|\;\|\nabla_d (u_h-Q_hu)\|\\
& &+\|\Pi_h(bu)-bQ_0u\|\;\|\nabla_d( u_h-Q_hu)\|
+\|c(u-Q_0u)\|\;\|u_0-Q_0u\|.
\end{eqnarray*}
Next, we use the G{\aa}rding's inequality (\ref{garding}) to obtain
\begin{eqnarray*}
\alpha_1(\|\nabla_d(e_h)\|^2+\|e_{h,0}\|^2)&\le&
\|\Pi_h(a\nabla u)-aR_h(\nabla u)\|\;\|\nabla_d (u_h-Q_hu)\|\\
& &+\|\Pi_h(bu)-bQ_0u\|\;\|\nabla_d( u_h-Q_hu)\|\\
& &+\|c(u-Q_0u)\|\;\|u_0-Q_0u\|+K\|u_0-Q_0u\|^2\\
&\le& \frac{\alpha_1}{2} (\|\nabla_d(u_h-Q_hu)\|^2+\|u_0-Q_0u\|^2)\\
& &+C\left(\|\Pi_h(a\nabla u)-aR_h(\nabla u)\|^2 +
\|\Pi_h(bu)-bQ_0u\|^2\right.\\
& & +\left.  \|c(u-Q_0u)\|^2\right)+K\|u_0-Q_0u\|^2,
\end{eqnarray*}
which implies the desired estimate (\ref{H1errorestimate}).
\end{proof}

\subsection{An estimate in $L^2(\Omega)$}

We use the standard duality argument to derive an estimate for the
error $u_h-Q_h u$ in the standard $L^2$ norm over domain $\Omega$.

\bigskip
\begin{lemma}\label{l2-error}
Assume that the dual of the problem (\ref{pde}) and (\ref{bc}) has
the $H^{1+s}$ regularity. Let $u\in H^{1}(\Omega)$ be the solution
(\ref{pde}) and (\ref{bc}), and $u_h$ be a weak Galerkin
approximation of $u$ arising from (\ref{WG-fem}) by using either the
weak finite element space $S_h(j,j+1)$ or $S_h(j,j)$. Let $Q_hu$ be
the $L^2$ projection of $u$ in the corresponding finite element
space (recall that it is locally defined). Then, there exists a
constant $C$ such that
\begin{eqnarray*}
\|Q_0u-u_0\|&&\le Ch^s\left( h\|f-Q_0f\|+\|\nabla u- R_h(\nabla
u)\| + \|a\nabla u - R_h(a\nabla u)\| +\|u-Q_0 u\|\right.\\
&& +\|bu-R_h(bu)\|+\left.
\|cu-Q_0(cu)\|+\|\nabla_d(Q_hu-u_h)\|\right),
\end{eqnarray*}
provided that the meshsize $h$ is sufficiently small.
\end{lemma}

\begin{proof}
Consider the dual problem of (\ref{pde}) and (\ref{bc}) which seeks
$w\in H_0^1(\Omega)$ satisfying
\begin{eqnarray}
-\nabla\cdot (a \nabla w)-b\cdot\nabla w+ cw &=& Q_0u-u_0\quad
\mbox{in}\;\Omega\label{dual}
%w&=&0\quad \mbox{on}\; \partial\Omega,\label{dual-BC}
\end{eqnarray}
The assumed $H^{1+s}$ regularity for the dual problem implies the
existence of a constant $C$ such that
\begin{equation}\label{reg}
\|w\|_{1+s}\le C\|Q_0u-u_0\|.
\end{equation}
Testing (\ref{dual}) against $Q_0u-u_0$ element by element gives
\begin{eqnarray}
\|Q_0u-u_0\|^2&=&(-\nabla\cdot (a \nabla w),\;Q_0u-u_0)-(b\cdot\nabla w,\;Q_0u-u_0)
+ (cw,\;Q_0u-u_0)\nonumber\\
&=&I+II+III,\label{m1}
\end{eqnarray}
where $I, II,$ and $III$ are defined to represent corresponding
terms. Let us estimate each of these terms one by one.

For the term $I$, we use the identity (\ref{4.200}) to obtain
\begin{eqnarray*}
I&=&(-\nabla\cdot (a\nabla w),Q_0u-u_0)= (\Pi_h (a\nabla
w),\nabla_d(Q_hu-u_h)).
\end{eqnarray*}
Recall that $\nabla_d(Q_hu)=R_h (\nabla u)$ with $R_h$ being a local
$L^2$ projection. Thus,
\begin{eqnarray}\nonumber
I&=& (\Pi_h (a\nabla w),\nabla_d(Q_hu-u_h))=(\Pi_h (a\nabla w),R_h \nabla u-\nabla_d u_h)\\
&=& (\Pi_h (a\nabla w),\nabla u-\nabla_d u_h)\nonumber\\
&=& (\Pi_h (a\nabla w)-a\nabla w,\nabla u-\nabla_d u_h) + (a\nabla
w,\nabla u-\nabla_d u_h).\label{yes.888}
\end{eqnarray}
The second term in the above equation above can be handled as
follows. Adding and subtracting two terms $(a\nabla_d Q_h w,
\nabla_d u_h)$ and $(a(\nabla w - R_h \nabla w), \nabla u)$ and
using the fact that $\nabla_d (Q_h u) =R_h (\nabla u)$ and the
definition of $R_h$, we arrive at
\begin{eqnarray}
(a\nabla w,\nabla u-\nabla_d u_h) &=& (a\nabla w, \nabla u)
-(a\nabla w, \nabla_d u_h) \nonumber\\
  &=& (a\nabla w, \nabla u) -(a\nabla_d Q_h w, \nabla_d u_h)- (a(\nabla w - R_h \nabla w), \nabla_d u_h)\nonumber\\
  &=& (a\nabla w, \nabla u)- (a\nabla_d Q_h w, \nabla_d u_h)- (a(\nabla w - R_h \nabla w),
  \nabla_d u_h - \nabla u)\nonumber\\
  & & - (a(\nabla w - R_h \nabla w), \nabla u)\nonumber\\
  &=&  (a\nabla w, \nabla u)- (a\nabla_d Q_h w, \nabla_d u_h) -
  (a(\nabla w - R_h \nabla w), \nabla_d u_h - \nabla u)\label{yes.889}\\
  & & - (\nabla w - R_h \nabla w, a\nabla u -R_h (a\nabla
  u)).\nonumber
\end{eqnarray}
Substituting (\ref{yes.889}) into (\ref{yes.888}) yields
\begin{eqnarray}\label{yes.termI}
I&=&(\Pi_h (a\nabla w)-a\nabla w,\nabla u-\nabla_d u_h)- (a(\nabla w - R_h \nabla w), \nabla_d u_h - \nabla u)\\
& &-(\nabla w - R_h \nabla w, a\nabla u -R_h (a\nabla u))+(a\nabla
w, \nabla u)- (a\nabla_d Q_h w, \nabla_d u_h).\nonumber
\end{eqnarray}

For the term $II$, we add and subtract $(\nabla_d(Q_h
w),\;b(Q_0u-u_0))$ from $II$ to obtain
\begin{eqnarray*}
II&=&-(b\cdot\nabla w,\;Q_0u-u_0)\\
&=&-(\nabla w-\nabla_d(Q_hw),\;b(Q_0u-u_0))-(\nabla_d(Q_h w),\;b(Q_0u-u_0))\\
&=&-(\nabla w-\nabla_d(Q_hw),\;b(Q_0u-u_0))-(\nabla_d(Q_h w),\;bQ_0u)+(\nabla_d(Q_h w),\;bu_0).
\end{eqnarray*}
In the following, we will deal with the second term on the right
hand side of the above equation. To this end, we use (\ref{4.88})
and the definition of $R_h$ and $Q_0$ to obtain
\begin{eqnarray*}
(\nabla_d(Q_h w),\;bQ_0u)&=&(\nabla_d(Q_h w)-\nabla w,\;bQ_0u)+(\nabla w,\;bQ_0u)\\
&=&(\nabla_d(Q_h w)-\nabla w,\;bQ_0u-bu)+(\nabla_d(Q_h w)-\nabla w,\;bu)\\
& &+(\nabla w,\;bQ_0u-bu)+(\nabla w,\;bu)\\
&=&(\nabla_d(Q_h w)-\nabla w,\;bQ_0u-bu)+(R_h(\nabla w)-\nabla w,\;bu-R_h(bu))\\
& &+(b\cdot\nabla w-Q_0(b\cdot\nabla w),\;Q_0u-u)+(\nabla w,\;bu).
\end{eqnarray*}
Combining the last two equations above, we arrive at
\begin{eqnarray}\label{yes.termII}
II&=&-(\nabla w-\nabla_d(Q_hw),\;b(Q_0u-u_0))-(\nabla_d(Q_h w)-\nabla w,\;bQ_0u-bu)\\
& &-(R_h(\nabla w)-\nabla w,\;bu-R_h(bu))-(b\cdot\nabla w-Q_0(b\cdot\nabla w),\;Q_0u-u)\nonumber\\
& &-(\nabla w,\;bu)+(\nabla_d(Q_h w),\;bu_0).\nonumber
\end{eqnarray}

As to the term $III$,  by adding and subtracting some terms and
using the fact that $Q_0$ is a local $L^2$ projection, we easily
obtain the following
\begin{eqnarray*}
III&=& (cw,\;Q_0u-u_0)=(cw-cQ_0w,\;Q_0u-u_0)+(cQ_0w,\;Q_0u-u_0)\\
&=& (cw-cQ_0w,\;Q_0u-u_0)+(cQ_0w,\;Q_0u)-(cQ_0w,\;u_0)\\
&=& (cw-cQ_0w,\;Q_0u-u_0)+(cQ_0w-cw,\;Q_0u)+(cw,\;Q_0u-u)\\
& &+(cw,\;u)-(cQ_0w,\;u_0)\\
&=& (cw-cQ_0w,\;Q_0u-u_0)+(Q_0w-w,\;cQ_0u-cu)+(Q_0w-w,\;cu-Q_0(cu))\\
& & +(cw-Q_0(cw),\;Q_0u-u)+(cw,\;u)-(cQ_0w,\;u_0).\\
\end{eqnarray*}

Note that the sum of the last two terms in $I$ (see
(\ref{yes.termI})), $II$ (see (\ref{yes.termII})), and $III$ (see
the last equation above) gives
\begin{eqnarray*}
(a\nabla w, \nabla u)- &&(a\nabla_d Q_h w, \nabla_d u_h)-(\nabla
w,\;bu)+(\nabla_d(Q_h w),\;bu_0)
+(cw,\;u)-(cQ_0w,\;u_0)\\
&&=a(u,w)-a(u_h, Q_hw)\\
&&=(f,\; w)-(f,\; Q_0w)\\
&&=(f-Q_0f,\;w-Q_0w).
\end{eqnarray*}
Thus, the sum of $I$, $II$, and $III$ can be written as follows:
\begin{eqnarray}
\|Q_0u-u_0\|^2=&& (f-Q_0f,w-Q_0w)+(\Pi_h (a\nabla w)-a\nabla w,\nabla u-\nabla_d u_h)\nonumber\\
&-& (a(\nabla w - R_h \nabla w), \nabla_d u_h - \nabla u)- ((\nabla w - R_h \nabla w), a\nabla u -R_h (a\nabla u))\nonumber\\
&-&(\nabla w-\nabla_d(Q_hw),\;b(Q_0u-u_0))-(\nabla_d(Q_h w)-\nabla w,\;bQ_0u-bu)\nonumber\\
&-&(R_h(\nabla w)-\nabla w,\;bu-R_h(bu))-(b\cdot\nabla w-Q_0(b\cdot\nabla w),\;Q_0u-u)\nonumber\\
&+&(cw-cQ_0w,\;Q_0u-u_0)+(Q_0w-w,\;cQ_0u-cu)\nonumber\\
&+&(Q_0w-w,\;cu-Q_0(cu))+(cw-Q_0(cw),\;Q_0u-u).\label{noname}
\end{eqnarray}
Using the triangle inequality, (\ref{4.88}) and (\ref{reg}), we can bound the second term
on the right hand side in the above equation by
\begin{eqnarray*}
\left|(\Pi_h (a\nabla w)-a\nabla w,\nabla u-\nabla_d
u_h)\right|&\leq&
\left|(\Pi_h (a\nabla w)-a\nabla w,\nabla u-\nabla_dQ_h u)\right|\\
& & +\left|(\Pi_h (a\nabla w)-a\nabla w,\nabla_dQ_h u-\nabla_d u_h)\right|\\
&\le&Ch^{s}\left(\|\nabla u-R_h(\nabla
u)\|+\|\nabla_d(Q_hu-u_h)\|\right)\|Q_0u-u_0\|.
\end{eqnarray*}
The other terms on the right hand side of (\ref{noname}) can be
estimated in a similar fashion, for which we state the results as
follows:
\begin{eqnarray*}
\left| (a(\nabla w - R_h \nabla w), \nabla_d u_h - \nabla u)\right|
&\le&Ch^{s}\left(\|\nabla u-R_h(\nabla
u)\|+\|\nabla_d(Q_hu-u_h)\|\right)\|Q_0u-u_0\|,\\
\left|((\nabla w - R_h \nabla w), a\nabla u -R_h (a\nabla u))\right|
&\le&Ch^{s}\|a\nabla u-R_h(a\nabla
u)\|\ \|Q_0u-u_0\|,\\
\left| (\nabla w-\nabla_d(Q_hw),\;b(Q_0u-u_0)) \right|
&\le&Ch^{s} \|Q_0u-u_0\|^2,\\
\left|(\nabla_d(Q_h w)-\nabla w,\;bQ_0u-bu)\right|
&\le&Ch^{s}\|u-Q_0u\|\ \|Q_0u-u_0\|,\\
\left|(R_h(\nabla w)-\nabla w,\;bu-R_h(bu))\right|
&\le&Ch^{s}\|bu-R_h(bu)\|\ \|Q_0u-u_0\|,\\
\left|(b\cdot\nabla w-Q_0(b\cdot\nabla w),\;Q_0u-u)\right|
&\le&Ch^{s}\|u-Q_0u\|\ \|Q_0u-u_0\|,\\
\left| (cw-cQ_0w,\;Q_0u-u_0) \right|
&\le&Ch\|Q_0u-u_0\|^2,\\
\left|(Q_0w-w,\;cQ_0u-cu)\right|
&\le&Ch\|u-Q_0u\|\ \|Q_0u-u_0\|,\\
\left|(Q_0w-w,\;cu-Q_0(cu))\right|
&\le&Ch\|cu-Q_0(cu)\|\ \|Q_0u-u_0\|,\\
\left|(cw-Q_0(cw),\;Q_0u-u)\right| &\le&Ch\|u-Q_0u\|\ \|Q_0u-u_0\|.
\end{eqnarray*}
Substituting the above estimates into (\ref{noname}) yields
\begin{eqnarray*}
\|Q_0u&&-u_0\|^2\le Ch^s\left( h\|f-Q_0f\|+\|\nabla u- R_h(\nabla
u)\| + \|a\nabla u - R_h(a\nabla u)\| +\|u-Q_0 u\|\right.\\
&& + \|bu-R_h(bu)\|+\left.
\|cu-Q_0(cu)\|+\|\nabla_d(Q_hu-u_h)\|+\|Q_0u-u_0\|\right)\|Q_0u-u_0\|.
\end{eqnarray*}
For sufficiently small meshsize $h$, we have
\begin{eqnarray*}
\|Q_0u-u_0\|&\le & Ch^s\left( h\|f-Q_0f\|+\|\nabla u- R_h(\nabla
u)\| + \|a\nabla u - R_h(a\nabla u)\| +\|u-Q_0 u\|\right.\\
&& +\|bu-R_h(bu)\|+\left.
\|cu-Q_0(cu)\|+\|\nabla_d(Q_hu-u_h)\|\right),
\end{eqnarray*}
which completes the proof.
\end{proof}

\subsection{Error estimates in $H^1$ and $L^2$}
With the results established in Lemma \ref{h1-error} and Lemma
\ref{l2-error}, we are ready to derive an error estimate for the
weak Galerkin approximation $u_h$. To this end, we may substitute
the result of Lemma \ref{l2-error} into the estimate shown in Lemma
\ref{h1-error}. If so, for sufficiently small meshsize $h$, we would
obtain the following estimate:
\begin{eqnarray*}
\|\nabla_d(u_h-Q_hu)\|^2+\|u_0-Q_0u\|^2 &\le& C\left(\|\Pi_h(a\nabla
u)-aR_h(\nabla u)\|^2 +\|c(u-Q_0u)\|^2
\right.\\
&& +\left. \|\Pi_h(bu)-bQ_0u\|^2\right) \\
&& + Ch^{2s}\left(h^2\|f-Q_0f\|^2+\|\nabla u- R_h(\nabla
u)\|^2\right.\\
&& + \left. \|a\nabla u - R_h(a\nabla u)\|^2 +\|u-Q_0 u\|^2\right.\\
&& +\left. \|bu-R_h(bu)\|^2+ \|cu-Q_0(cu)\|^2\right).
\end{eqnarray*}
A further use of the interpolation error estimate leads to the
following error estimate in a discrete $H^1$ norm.

\begin{theorem}\label{H1error-estimate}
In addition to the assumption of Lemma \ref{l2-error}, assume that
the exact solution $u$ is sufficiently smooth such that $u\in
H^{m+1}(\Omega)$ with $0\le m \le j+1$. Then, there exists a
constant $C$ such that
\begin{eqnarray}\label{trueH1error}
\|\nabla_d(u_h-Q_hu)\|+\|u_0-Q_0u\| \le C(h^{m} \|u\|_{m+1}
+h^{1+s}\|f-Q_0f\|).
\end{eqnarray}
\end{theorem}

\medskip

Now substituting the error estimate (\ref{trueH1error}) into the
estimate of Lemma \ref{l2-error}, and then using the standard
interpolation error estimate we obtain
\begin{eqnarray*}
\|u_h-Q_hu\|&\le& C\left(h^{1+s}\|f-Q_0f\|+h^{m+s} \|u\|_{m+1} +
h^s(h^{m} \|u\|_{m+1} +h^{1+s}\|f-Q_0f\|)\right)\\
&\le&C\left(h^{1+s}\|f-Q_0f\|+h^{m+s} \|u\|_{m+1}\right).
\end{eqnarray*}
The result can then be summarized as follows.

\bigskip
\begin{theorem}\label{trueL2error} Under the assumption of Theorem
\ref{H1error-estimate}, there exists a constant $C$ such that
\begin{eqnarray*}
\|u_h-Q_hu\| \le C\left(h^{1+s}\|f-Q_0f\|+h^{m+s}
\|u\|_{m+1}\right), \quad s\in (0,1],\ m\in (0,j+1],
\end{eqnarray*}
provided that the mesh-size $h$ is sufficiently small.
\end{theorem}

\medskip

If the exact solution $u$ of (\ref{pde}) and (\ref{bc}) has the
$H^{j+2}$ regularity, then we have from Theorem \ref{trueL2error}
that
\begin{eqnarray*}
\|u_h-Q_hu\| &\le& C\left(h^{1+s}h^{j}\|f\|_{j}+h^{j+s+1}
\|u\|_{j+2}\right)\\
 &\le& C h^{j+s+1} \left(\|f\|_{j}+\|u\|_{j+2}\right)
\end{eqnarray*}
for some $0<s\leq 1$, where $s$ is a regularity index for the dual
of (\ref{pde}) and (\ref{bc}). In the case that the dual has a full
$H^2$ (i.e., $s=1$) regularity, one would arrive at
\begin{eqnarray}\label{superc}
\|u_h-Q_hu\| \le C h^{j+2} \left(\|f\|_{j}+\|u\|_{j+2}\right).
\end{eqnarray}
Recall that on each triangular element $T^0$, the finite element
functions are of polynomials of order $j\ge 0$. Thus, the error
estimate (\ref{superc}) in fact reveals a superconvergence for the
weak Galerkin finite element approximation arising from
(\ref{WG-fem}).

\end{document}